\documentclass{article}
\usepackage[letterpaper, left=1in, right=1in, top=1in, bottom=1in]{geometry}

\PassOptionsToPackage{hypertexnames=false}{hyperref}  
\usepackage[colorlinks=true, linkcolor=blue!70!black, citecolor=blue!70!black,urlcolor=black,breaklinks=true]{hyperref}
\usepackage{microtype}
\usepackage{hhline}

\makeatletter
\newcommand{\neutralize}[1]{\expandafter\let\csname c@#1\endcsname\count@}
\makeatother

\usepackage{algorithm}
\usepackage{amsthm}
\usepackage{mathtools}
\usepackage{amsmath}
\usepackage{bbm}
\usepackage{amsfonts}
\usepackage{amssymb}

\usepackage{xpatch}

\newtheorem{theorem}{Theorem}
\newtheorem{question}{Question}
\newtheorem{lemma}{Lemma}

\newtheorem{conjecture}{Conjecture}

\theoremstyle{definition}
\newtheorem{example}{Example}

\theoremstyle{remark}
\newtheorem{remark}{Remark}
\AtBeginEnvironment{remark}{%
  \pushQED{\qed}%
}
\AtEndEnvironment{remark}{\popQED\endexample}

\makeatletter
  \renewenvironment{proof}[1][Proof]%
  {%
   \par\noindent{\bfseries\upshape {#1.}\ }%
  }%
  {\qed\newline}
  \makeatother

\xpatchcmd{\proof}{\itshape}{\normalfont\proofnameformat}{}{}
\newcommand{\proofnameformat}{\bfseries}

\usepackage[nameinlink,capitalize]{cleveref}

\crefformat{equation}{#2(#1)#3}
\Crefformat{equation}{#2(#1)#3}

\Crefformat{figure}{#2Figure #1#3}
\Crefname{assumption}{Assumption}{Assumptions}
\Crefformat{assumption}{#2Assumption #1#3}

\usepackage{crossreftools}
\usepackage{xparse}

\ExplSyntaxOn
\DeclareDocumentCommand{\XDeclarePairedDelimiter}{mm}
 {
  \__egreg_delimiter_clear_keys: 
  \keys_set:nn { egreg/delimiters } { #2 }
  \use:x 
   {
    \exp_not:n {\NewDocumentCommand{#1}{sO{}m} }
     {
      \exp_not:n { \IfBooleanTF{##1} }
       {
        \exp_not:N \egreg_paired_delimiter_expand:nnnn
         { \exp_not:V \l_egreg_delimiter_left_tl }
         { \exp_not:V \l_egreg_delimiter_right_tl }
         { \exp_not:n { ##3 } }
         { \exp_not:V \l_egreg_delimiter_subscript_tl }
       }
       {
        \exp_not:N \egreg_paired_delimiter_fixed:nnnnn 
         { \exp_not:n { ##2 } }
         { \exp_not:V \l_egreg_delimiter_left_tl }
         { \exp_not:V \l_egreg_delimiter_right_tl }
         { \exp_not:n { ##3 } }
         { \exp_not:V \l_egreg_delimiter_subscript_tl }
       }
     }
   }
 }

\keys_define:nn { egreg/delimiters }
 {
  left      .tl_set:N = \l_egreg_delimiter_left_tl,
  right     .tl_set:N = \l_egreg_delimiter_right_tl,
  subscript .tl_set:N = \l_egreg_delimiter_subscript_tl,
 }

\cs_new_protected:Npn \__egreg_delimiter_clear_keys:
 {
  \keys_set:nn { egreg/delimiters } { left=.,right=.,subscript={} }
 }

\cs_new_protected:Npn \egreg_paired_delimiter_expand:nnnn #1 #2 #3 #4
 {
  \mathopen{}
  \mathclose\c_group_begin_token
   \left#1
   #3
   \group_insert_after:N \c_group_end_token
   \right#2
   \tl_if_empty:nF {#4} { \c_math_subscript_token {#4} }
 }
\cs_new_protected:Npn \egreg_paired_delimiter_fixed:nnnnn #1 #2 #3 #4 #5
 {
  \mathopen{#1#2}#4\mathclose{#1#3}
  \tl_if_empty:nF {#5} { \c_math_subscript_token {#5} }
 }
\ExplSyntaxOff

\XDeclarePairedDelimiter{\supnorm}{
  left=\lVert,
  right=\rVert,
  subscript=\infty
  }

\usepackage[utf8]{inputenc} 
\usepackage[T1]{fontenc}    
\usepackage{url}            
\usepackage{booktabs}       
\usepackage{amsfonts}       
\usepackage{nicefrac}       
\usepackage{microtype}      
\usepackage{authblk}
\usepackage{tocloft}            

\usepackage{enumitem}

\usepackage{etoolbox}
\usepackage{comment}
\newtoggle{draft}
\togglefalse{draft}

\usepackage{mathrsfs}

\usepackage{algorithm}
\usepackage{verbatim}
\usepackage[noend]{algpseudocode}

\usepackage{multicol}

\usepackage{colortbl}
\usepackage{bbm}
\usepackage{setspace}

\usepackage{transparent}

\usepackage{inconsolata}
\usepackage[scaled=.90]{helvet}
\usepackage{xspace}

\usepackage{pifont}

\usepackage{tikz-cd}

\usepackage{hyperref}

\newcommand{\R}{\mathbb{R}} 
\newcommand{\E}{\mathbb{E}} 
\newcommand{\Sp}{\mathbb{S}^{d-1}}
\DeclareMathOperator{\vol}{vol}
\newcommand{\NN}{N^{-\frac{2}{d-1}}}

\DeclareMathOperator{\argmin}{argmin}

\def\ddefloop#1{\ifx\ddefloop#1\else\ddef{#1}\expandafter\ddefloop\fi}
\def\ddef#1{\expandafter\def\csname c#1\endcsname{\ensuremath{\mathcal{#1}}}}
\ddefloop ABCDEFGHIJKLMNOPQRSTUVWXYZ\ddefloop
\begin{document}
	
	\title{On the Optimality of  Random Partial Sphere Coverings \\in High Dimensions}
	
	\author{Steven Hoehner and Gil Kur}
	
	\renewcommand{\thefootnote}{\fnsymbol{footnote}} 
	\renewcommand{\thefootnote}{\arabic{footnote}} 
	
	\maketitle

	\begin{abstract}
		Given $N$ geodesic caps on the unit sphere in $\mathbb{R}^d$, whose total normalized surface area is one, what is the maximal proportion of the sphere that their union can cover? In this work, we provide an asymptotically sharp upper bound  for an antipodal partial covering of the sphere by $N=N(d)$ congruent caps in the regime
$N(d)\to\infty$ and $\ln N(d)=o(\sqrt d)$, showing that the maximum proportion covered approaches $1 - e^{-1}$ as $d\to\infty$. We discuss the relation of this result to the optimality of random polytopes in high dimensions,  the limitations of our technique via the Gaussian surface area bounds of K. Ball and F. Nazarov, and its applications in computer science theory.
	\end{abstract}

	\section{Introduction and Main Results}

	In this note, we study the following partial covering problem for the sphere:
	
	\begin{question}
		Given $N$ geodesic caps on the unit sphere $\mathbb{S}^{d-1}\subset\mathbb{R}^d$, each covering a proportion $1/N$ of the surface area, what is the largest proportion of the sphere that can be covered?
	\end{question}
	
	This question traces back to a classical paper of Erd\H{o}s, Few and Rogers \cite{EFR}\footnote{Their formulation is  equivalent to ours when $N\ge c(d)$ for some sufficiently large $c(d)$; see Remark~\ref{R:EFR}.}. Variants of the problem have appeared in several contexts, including the study of random polytopes and geometric probability. In particular, the problem was highlighted in \cite{HK-DCG} in connection with the optimality of random polytopes; it was also posed by Glazyrin \cite{Glazyrin} and discussed by Aubrun and Szarek in their blog post \cite{ABMB-blog-post}.
	
	Even when both $d$ and $N$ are large, it is far from clear that one can arrange many such caps with little overlap.  For example, as discussed in \cite{ABMB-blog-post}, if one considers the packing problem for equal caps of the same geodesic
radius as a cap of measure $1/N$, then a theorem of Kabatjanski\u{\i} and Leven\v{s}te\u{\i}n \cite{KL-1978} implies that the maximal packing density is
at most $(2/3)^d$. Thus, disjoint packings are much too sparse to explain the
positive covering fraction obtained by the random placement of caps. 
	
	Motivated by the optimality of random constructions in high dimensions (see, e.g.,  \cite{alon2016probabilistic,AGA_Part1} and references within),   in \cite{HK-DCG} it was shown that if $N \geq 1$ and the centers of the caps are chosen uniformly and independently from the sphere, then with high probability, the proportion of the sphere that is covered is  \textbf{at least}  $1-e^{-1}+O(N^{-1})$, or  about 63.2\% of the sphere. Surprisingly, this estimate  is dimension-free. This naturally raises the question: \emph{Is the random configuration asymptotically optimal as $d,N\to\infty$?} A complementary probabilistic refinement of this random model was recently obtained in \cite{Hoehner-Thale-CLT}, where a central limit theorem and a quantitative Kolmogorov distance bound were proved for the covered volume of random partial sphere coverings. In the opposite direction, Erd\H{o}s, Few and Rogers \cite{EFR} showed that for $d\ge 2^{20}$, even with an  optimal deterministic configuration, the covered proportion $\vartheta_d(1)$ cannot exceed approximately $0.92334$ (see Remark \ref{R:EFR} below). Thus, the optimal high-dimensional covering fraction remains somewhere between $0.632$ and $0.92334$.
	
	
	In this note, we establish the asymptotic optimality of the random configuration under an additional but natural symmetry assumption, namely, that the caps occur in \emph{antipodal pairs}.  The antipodality assumption is particularly useful because it allows us to invoke the \emph{Gaussian Correlation Inequality} (GCI). In our case, it suffices to use a weaker version of the GCI, known as \v{S}id\'ak's inequality, since 
	our approach relates the spherical covering problem to Gaussian measures of symmetric slabs in $\mathbb{R}^d$. Exploiting the thin-shell concentration of the Gaussian measure around $\sqrt{d}\,\mathbb{S}^{d-1}$, we choose the slab thickness to correspond  to the height of the spherical caps. We then prove that when the number of caps grows at rate of $\exp(o(\sqrt{d}))$, the random antipodal configuration achieves the optimal asymptotic coverage.
	
	Given $d,N\ge 3$ where $N$ is even, consider an antipodal configuration $\pm x_1,\ldots,\pm x_{N/2}\in\mathbb{S}^{d-1}$. Let $C(x_i)$ denote the spherical cap centered at $x_i$ such that $\sigma_{d-1}(C(x_i))=N^{-1}$, where $\sigma_{d-1}$ is the normalized $(d-1)$-dimensional surface measure on $\mathbb{S}^{d-1}$.  Define
	
	\begin{equation}\label{VNd-main}
	V_N(d):=\max_{\pm x_1,\ldots,\pm x_{N/2}\in\mathbb{S}^{d-1}}
	\sigma_{d-1}\!\left(\bigcup_{i=1}^{N/2}\big[C(x_i)\cup C(-x_i)\big]\right),
	\end{equation}
	which is the optimal covered proportion of the sphere by antipodal pairs of equal caps, each cap having measure $1/N$ and each pair having total measure $2/N$. Our main result is the following
	
	\begin{theorem}\label{mainThm-corrected}
		Let $N=N(d)$ be an even integer sequence such that $N(d)\to\infty$ and $\ln N(d)=o(\sqrt d)$. Then,
		\[
		\lim_{d\to\infty} V_N(d) = 1 - e^{-1}.
		\]
	\end{theorem}
	Theorem~\ref{mainThm-corrected} thus provides a new geometric interpretation of Euler’s number $e$:
	\begin{equation}\label{Eulers}
		e = \frac{1}{1 - \lim_{d\to\infty} V_N(d)}.
	\end{equation}
	Equation \eqref{Eulers} shows that, like $\pi$, Euler's number $e$ may also be defined in terms of the geometry of the sphere in high dimensions.  Theorem \ref{mainThm-corrected} also confirms \cite[Conjecture 4.2]{HK-DCG} in the case of antipodal partial coverings by $N=N(d)$  caps satisfying $N(d)\to\infty$ and $\ln N(d)=o(\sqrt d)$, showing that, under these assumptions, random partial coverings are indeed asymptotically optimal (up to a $1+o(1)$ factor). 
    
    To establish the equality in Theorem \ref{mainThm-corrected}, we show the inequality in both directions. For the lower bound, we use random antipodal partial coverings. We follow the arguments in  \cite{HK-DCG}, with slight modifications made since our partial coverings are antipodal. We refer the reader to \cite[Subsection 9.1]{HK-DCG} for the details. Thus, the main contribution of this paper is the upper bound in Theorem \ref{mainThm-corrected}.
    
    Nevertheless, the present paper leaves several open directions. In particular, we do not know how to remove the antipodality assumption, or how to extend the result to the regime $N=e^{\Omega(\sqrt{d})}$. Our method also does not yet resolve whether $V_N(d)$ decreases monotonically to $1-e^{-1}$, 
	a property that would describe how overlaps grow with increasing dimension.
	
	\begin{remark}\label{R:EFR}
		In \cite{EFR}, Erd\H{o}s, Few and Rogers study the following asymptotic version of the partial covering problem.  For each $r>0$, consider a finite family $\Sigma_r$ of congruent caps on $r\mathbb{S}^{d-1}$, and suppose that 
        \[
\sum_{C\in\Sigma_r}\operatorname{vol}_{d-1}(C)
   \leq\delta\operatorname{vol}_{d-1}(r\mathbb{S}^{d-1}).
        \]
        Define
		\[
        \vartheta_d(\delta)= 
 \limsup_{r\to\infty}\sup_{\Sigma_r}
 \frac{\operatorname{vol}_{d-1}\big(\bigcup_{C\in\Sigma_r} C\big)}
      {\operatorname{vol}_{d-1}(r\mathbb{S}^{d-1})},
		\]
        where the supremum is taken over all finite families $\Sigma_r$
of congruent caps satisfying the preceding density constraint. 
		In the case $\delta=1$, Erd\H{o}s, Few and Rogers prove that if $d\geq 2^{20}$ and $\sum_{C\in\Sigma_r} \operatorname{vol}_{d-1}(C)=\operatorname{vol}_{d-1}(r\mathbb{S}^{d-1})$, then $\vartheta_d(1)$ cannot exceed  $0.92334$ (approximately). Our result provides an asymptotic lower bound on this quantity (see also \cite{HK-DCG}), showing that $\liminf_{d\to\infty}\vartheta_d(1) \geq 1 - e^{-1} \approx 0.632$.
	\end{remark}
	
	The proof of Theorem \ref{mainThm-corrected} is given in Section \ref{main-proof-sec}. A key ingredient is \v{S}id\'ak's Lemma (see Lemma \ref{Sidaks-lemma} below), which states that origin-symmetric slabs are positively correlated in Gaussian space $(\R^d,\gamma_d)$. \v{S}id\'ak's Lemma is a special case of the celebrated Gaussian Correlation Inequality due to Royen \cite{royen}, which states that all origin-symmetric convex bodies are positively correlated in $(\R^d,\gamma_d)$. For recent developments related to the Gaussian Correlation Inequality and its refinements, we refer the reader to the recent work of Milman \cite{Milman-GCI-IBL} and Milman, Nakamura and Tsuji \cite{milman2026gaussian}, and to the refinement of the \v{S}id\'ak--Khatri inequality and related strengthened Gaussian correlation conjectures in \cite{ACS-Sidak-Khatri}.
	
	\subsection{Notation and Definitions}\label{Backgroundsec}
	
	 For vectors $x=(x_1,\ldots,x_d),y=(y_1,\ldots,y_d)\in\R^d$, we use the standard inner product $\langle x,y\rangle=\sum_{i=1}^d x_i y_i$, and the Euclidean norm of $x=(x_1,\ldots,x_d)\in\R^d$ is $\|x\|=\sqrt{\langle x,x\rangle}=\sqrt{\sum_{i=1}^d x_i^2}$. 
	The $d$-dimensional Euclidean unit ball centered at the origin is denoted $B_d$, and its boundary  $\partial B_d=\Sp$ is the unit sphere in $\R^d$ centered at the origin. The $d$-dimensional volume of $B_d$ is $\vol_d(B_d)=\frac{\pi^{d/2}}{\Gamma\left(\frac{d}{2}+1\right)}$, where for $x>0$,  $\Gamma(x)=\int_0^\infty t^{x-1}e^{-t}\,dt$ is the Gamma function. By the cone-volume formula, we have $\vol_{d-1}(\partial B_d)=d\vol_d(B_d)$. One may estimate the volume of the $d$-dimensional Euclidean ball using Stirling's inequality 
	\begin{equation}\label{stirling}
		\sqrt{2\pi x}\left(\frac{x}{e}\right)^x\leq\Gamma(x+1) \leq  \sqrt{2\pi x}\left(\frac{x}{e}\right)^x e^{\frac{1}{12x}},\qquad x\geq 1.
	\end{equation}
  Throughout the paper, the notation $f(d)\sim g(d)$ means that the  functions $f(d)$ and $g(d)$ are asymptotically equivalent, i.e., $\lim_{d\to\infty}\frac{f(d)}{g(d)}=1$ (and hence $f(d)=(1+o(1))g(d)$).
    
	 Recall that for a Borel set $A\subset \R^d$, the Gaussian measure $\gamma_d$ is defined by
	\[
	\gamma_d(A) = \frac{1}{(2\pi)^{d/2}}\int_A e^{-\|x\|^2/2}\,dx.
	\]
	We denote the standard normal cumulative distribution function by $\Phi(z)=\gamma_1((-\infty,z])$. The main ingredient we need to prove the upper bound in Theorem \ref{mainThm-corrected} is \v{S}id\'ak's Lemma (see, e.g., \cite[Lemma 2]{Ball-2001}). A \emph{symmetric slab} in $\R^d$ is a set of the form $\{x\in\R^d:\,|\langle x,u\rangle|\leq t\}$, where $u\in\mathbb{S}^{d-1}$ and $t>0$.
	
	\begin{lemma}[\v{S}id\'ak's Lemma]\label{Sidaks-lemma}
		Let $m\in\mathbb{N}$. If $K_1,\ldots,K_m$ are symmetric slabs in $\R^d$, then
		\begin{equation}\label{Sidak-ineq}
			\gamma_d\left(\bigcap_{i=1}^m K_i\right)\geq \prod_{i=1}^m \gamma_d(K_i).
		\end{equation}
	\end{lemma}
    
	\subsection{Proof Outline}
The main idea of the proof of Theorem \ref{mainThm-corrected} is to replace spherical geometry by Gaussian geometry. This is natural because the Gaussian measure is rotationally invariant, and because the uniform measure on the sphere arises as the angular part of a Gaussian vector. Let $Y\sim N(0,I_d)$, and write $Y=RU$, where $R=\|Y\|$ and $U=Y/\|Y\|\in\Sp$. Then $U$ is uniformly distributed on the sphere and is independent of $R$.

A spherical cap on $\Sp$ therefore corresponds naturally to a cone in $\R^d$: given a cap $C\subset\Sp$, we define its associated cone $A=\{ru:\,r\geq 0, u\in C\}$. The probability that a random direction lies in the cap $C$ is exactly the probability that $Y$ lies in the cone $A$. By rotational invariance, we may assume that the center of the cap is the north pole $e_d$. Accordingly, the cone has axis $\R e_d$. We define $t_{d,N}$ to be the Gaussian quantile satisfying 
\[
\Phi(-t_{d,N})=\gamma_1([t_{d,N},\infty))=\frac{1}{N}.
\]
Geometrically, $t_{d,N}$ is the height of the halfspace $H^+=\{x\in\R^d:\,\langle x,e_d\rangle\geq t_{d,N}\}$ with Gaussian measure $1/N$. Our eventual goal is to apply \v{S}id\'ak's lemma, which applies to symmetric slabs (or, equivalently, complements of halfspaces). To do so, we need to relate the cone $A$ to a Gaussian halfspace of height approximately $t_{d,N}$. 

    Two quantities play a central role: 
	\begin{itemize}
		\item $A_{d,N}$: the height of a spherical cap on  $\sqrt{d} \cdot \mathbb{S}^{d-1}$ whose normalized surface area is $1/N$; and
		\item $t_{d,N}$, the Gaussian quantile defined above. 
	\end{itemize}
    Although these arise from different geometric settings, they turn out to be very close after the natural scaling by $\sqrt{d}$.  
	\begin{lemma}\label{Lem:Caps}
		Let  $N:=N(d)\to\infty$, and write $f(d)=\ln N(d)$. Assume that  $f(d)\to\infty$ and $f(d)/\sqrt{d}\to 0$ as $d\to\infty$. Then
		\[
		\delta_{d,N}:= |t_{d,N} - A_{d,N}| = \frac{f(d)^{3/2}}{\sqrt{2}\,d}
		+o\!\left(\frac{1}{\sqrt{f(d)}}+\frac{f(d)^{3/2}}{d}\right) = o(f(d)^{-1/2}).
		\]
		Moreover,
        \[
        t_{d,N} =
		\sqrt{2f(d)}
		-\frac{\ln(4\pi f(d))}{2\sqrt{2f(d)}} + o(f(d)^{-1/2}).
        \]
	\end{lemma}
	This lemma shows that spherical caps of area $1/N$ and Gaussian halfspaces of measure $1/N$ have nearly the same height, up to the small error $\delta_{d,N}$. The proof combines the standard Gaussian tail quantile asymptotics with a direct integration by parts estimate for the spherical cap integral.
    
    To compare the cone $A$ with a Gaussian halfspace, we ``trim'' the cone at some height above $t_{d,N}$. The idea is to discard a negligible portion of the cone near its tip, while ensuring that the remaining part behaves like a halfspace of measure $1/N$. We define the trimmed height 
    \[
    t_{d,N}' := t_{d,N}+s_d
    \]
    where $s_d>0$ is chosen so that $\sqrt{f(d)}\cdot s_d\to 0$, $\frac{\sqrt{d}}{\sqrt{f(d)}}\cdot s_d\to\infty$, and $|t_{d,N}-A_{d,N}|=o(s_d)$ as $d\to\infty$. Such a choice is possible because, by Lemma \ref{Lem:Caps}, $|t_{d,N}-A_{d,N}|=o(f(d)^{-1/2})$, while $\sqrt{f(d)/d}=o(f(d)^{-1/2})$. The first condition ensures that the one-dimensional Gaussian tail does not change, i.e.,
    \[
    \gamma_1([t_{d,N}',\infty))=(1+o(1))\gamma_1([t_{d,N},\infty))=\frac{1+o(1)}{N}.
    \]
    The second condition ensures that the geometry of the cone is favorable in high dimensions. 

    Using Fubini's theorem, the Gaussian measure of the trimmed cone can be written as
    \[
\gamma_d(A\cap\{x_d\geq t_{d,N}'\}) =\frac{1}{\sqrt{2\pi}}\int_{t_{d,N}'}^\infty e^{-t^2/2}\Pr(\|Y'\|\leq t\, r_{d,N})\,dt
    \]
    where $Y'\sim N(0,I_{d-1})$ and $r_{d,N}$ is the radius of the $(d-1)$-dimensional cross-section of the cone at height 1. Since
\[
r_{d,N}=\frac{\sqrt{1-\alpha_{d,N}^2}}{\alpha_{d,N}}
=(1+o(1))\sqrt{\frac{d}{2f(d)}},
\]
we have
\[
s_d r_{d,N}=(1+o(1))s_d\sqrt{\frac{d}{2f(d)}}.
\]
Let $\eta_d:=t_{d,N}-A_{d,N}$. By the choice of $s_d$, we have $|\eta_d|=o(s_d)$. Hence,
\[
\frac{t_{d,N}+s_d}{A_{d,N}}
=\frac{t_{d,N}+s_d}{t_{d,N}-\eta_d}
=1+(1+o(1))\frac{s_d}{t_{d,N}}
=1+\Theta\!\left(\frac{s_d}{\sqrt{f(d)}}\right).
\]
Since $\alpha_{d,N}^2=O(f(d)/d)$, it follows that
\[
(t_{d,N}+s_d)r_{d,N}
=\sqrt d+\Theta\!\left(\frac{\sqrt d}{\sqrt{f(d)}}s_d\right)+o(1).
\]
By the assumption $s_d\sqrt{d/f(d)}\to\infty$, this implies
\[
(t_{d,N}+s_d)r_{d,N}\geq\sqrt{d}+\omega(1).
\]
Since $t\,r_{d,N}$ increases with $t$, the same lower bound holds for every $t\geq t_{d,N}'$. Thus,
\[
t\,r_{d,N}\geq\sqrt{d}+\omega(1) \quad\text{uniformly for all }t\geq t_{d,N}'.
\]
    By Gaussian concentration for Lipschitz functions (applied to the norm $\|Y'\|$), the random variable $\|Y'\|$ is sharply concentrated around $\sqrt{d-1}$. Consequently,
    \[
\Pr(\|Y'\|\leq t\,r_{d,N}) =1-o(1)\quad\text{uniformly for all }t\geq t_{d,N}'.
    \]
    Substituting this into the Fubini integral yields
    \[
\gamma_d(A\cap\{x_d\geq t_{d,N}'\}) = (1+o(1))\gamma_1([t_{d,N}',\infty))=\frac{1+o(1)}{N}.
    \]
    This establishes the desired equivalence between the trimmed cone and the Gaussian halfspace.

    Our reduction shows that, after trimming at height $t_{d,N}'$, the cone generated by a spherical cap of area $1/N$ behaves like a Gaussian halfspace of measure $1/N$. This finally allows us to apply \v{S}id\'ak's lemma to symmetric slabs and transfer the resulting bounds back to the original spherical problem. The competing requirements on the trimming parameter $s_d$ explain why our method applies precisely in the regime $f(d)=o(\sqrt{d})$, and why this approach cannot be extended beyond that scale.
	\section{Discussion}
	\paragraph{On the Threshold of $\ln N  = o(\sqrt{d})$:}
	One may ask whether the restriction $\ln N=o(\sqrt d)$ is merely an artifact of accumulating error terms in the order-statistic and integration-by-parts estimates, or whether it reflects a genuine limitation of the Gaussian reduction. The proof of our main result is based on   special properties of the Gaussian measure which are not true for the unit sphere and the volume measure. Specifically, we propose the following explanation via  \emph{Gaussian surface area} (also called \emph{Gaussian perimeter}).  For any Borel set $K\subset\mathbb{R}^d$, it is defined as
\[
\operatorname{GSA}(K)
\;:=\;
\liminf_{\varepsilon\downarrow 0}\,
\frac{\gamma_d\!\big(K+\varepsilon B_d\big)-\gamma_d(K)}{\varepsilon}.
\]	
The seminal results of \cite{ball1993reverse,nazarov2004maximal} showed a sharp bound for $\operatorname{GSA}(K)$ of the order  $O(d^{1/4})$ when $K$ is convex and symmetric. 
By the recent result of \cite{de2024gaussian} (see also \cite{klivans2008learning} and references within), one can approximate  $K = (\sqrt{d} + O(1))B_d$ with $\exp(C\sqrt{d})$ facets up to an absolute constant accuracy, with respect to the Gaussian measure of the set difference. 

This bound is in sharp contrast to the volume measure that requires at least $\exp(cd)$ facets to obtain a constant approximation under the symmetric difference metric.  As our proof is based on random approximation of the ball via a symmetric polytope, in a Gaussian space, one may apply such a reduction only when $\ln N = o(\sqrt{d})$; as $\ln N \gtrsim \sqrt{d}$, we cannot use the properties of Gaussian measure, where the approximation rate behaves differently from that of the volume measure. Roughly speaking, in this regime, the ball and the polytopes look the same under Gaussian measure. Therefore, to extend Theorem \ref{mainThm-corrected} to the $\ln N = \Omega(\sqrt{d})$ regime, one needs to find a different approach. 

Another question is how to relax the symmetric covering assumption. We believe that it is an artifact of \v{S}id\'ak's inequality, and is not essential to the underlying argument.
	
	\paragraph{On the Optimality of Random Polytopes:}
		The approximation of convex bodies by \emph{arbitrarily positioned} polytopes with $N$ facets (or $N$ vertices) has recently been studied; see, e.g., \cite{BH-2024,BHK,Boroczky2004,GTW2021,GW2018,HSW,HSW-2025,Kur2017,ludwig1999,LSW}. In this model, there is no restriction on the relative positions of the body $K$ and the approximating polytope $P$.  This approach is mainly motivated by the fact that, when the polytope is in arbitrary position, the rate of approximation is significantly faster than in models imposing a strict inclusion. In fact, dropping the restriction improves the estimate by a factor of dimension, and requires only exponentially many facets, rather than superexponentially many. For more background, we refer the reader to, e.g., \cite{HK-DCG,Kur2017,LSW} and the references therein. 
	
	In the arbitrary position approximation model,  for every fixed convex body $K\subset\R^d$ and every  $N\geq d+1$, there exists a best-approximating polytope
	\begin{equation*}
		P_{K,N}^{\rm b}\in \argmin\left\{\vol_d(P \triangle K):\,P\subset\R^d \text{ is a polytope with at most $N$ facets}\right\},
	\end{equation*}
	where the ``b'' superscript indicates a best-approximating polytope in arbitrary position. In this setting, Ludwig \cite{ludwig1999} proved that for any $C^2$ convex body $K$ in $\R^d$ with positive Gaussian curvature $\kappa$,
	\begin{equation}\label{Eq:ldiv}
		\lim_{N\to\infty}\frac{\vol_d(P_{K,N}^{\rm b}\triangle K)}{\NN}=\frac{{\rm ldiv}_{d-1}\cdot{\rm as}(K)^{\frac{d+1}{d-1}}}{2},
	\end{equation}
	where ${\rm as}(K)=\int_{\partial K}\kappa(x)^{\frac{1}{d+1}}\,d\mu_{\partial K}(x)$ is the affine surface area of $K$ (see, e.g., \cite{SchneiderBook,SW-affine-SA}) and $\mu_{\partial K}$ is the usual surface area measure of $K$, and ${\rm ldiv}_{d-1}$ is the \emph{Laguerre--Dirichlet--Voronoi tiling number} in $\R^{d-1}$. 
 For general $d\geq 4$, the exact value of ${\rm ldiv}_{d-1}$ is unknown \cite{BL-1999,ludwig1999}. Ludwig, Sch\"utt and Werner \cite{LSW}  showed that ${\rm ldiv}_{d-1}\geq c_1$ for some absolute constant $c_1 > 0$ (see also \cite{Kur2017}).  On the other hand, Kur \cite{Kur2017} used  random polytopes to show that
	\begin{equation}\label{Eq:ldivupper}
		{\rm ldiv}_{d-1} \leq (\pi e)^{-1}\left(\int_{0}^{1}t^{-1}(1-2^{-t})\,dt+\int_{0}^{\infty}2^{-e^{t}}\,dt\right) + o_d\left(1\right).
	\end{equation}	
	
	The partial covering problem for $\Sp$ is closely related to the problem of approximating the Euclidean unit ball $B_d$ by arbitrarily positioned  polytopes with $N$ facets; see \cite{Matousek-2002}, \cite{Boroczky-Wintsche-2003} and \cite{HK-DCG}. In light of the results in this paper, we believe that the upper bound \eqref{Eq:ldivupper} is optimal.
	
	\begin{conjecture}\label{C:RndPoly}
		The following estimate holds for ${\rm ldiv}_{d-1}$ as defined in \eqref{Eq:ldiv}:
		\[
		{\rm ldiv}_{d-1} = (\pi e)^{-1}\left(\int_{0}^{1}t^{-1}(1-2^{-t})\,dt+\int_{0}^{\infty}2^{-e^{t}}\,dt\right) + o_d\left(1\right).
		\]
	\end{conjecture}
	If true, Conjecture \ref{C:RndPoly} would imply that the optimal approximation of the Euclidean ball is achieved  (on average) by  an intersection of $N=e^{f(d)}$  random slabs  up to an error of $o_d(1)$, meaning random constructions are  optimal in the arbitrary setting.

	\section{Proof of Theorem \ref{mainThm-corrected}}\label{main-proof-sec}
    
The lower bound was proved in \cite[Subsection 9.1]{HK-DCG}  via a random construction. It remains to prove the upper bound. To begin, fix some $x\in\Sp$. The cap $C(x)$ centered at $x$ of measure $1/N$ can be written as
\[
C(x)=\{u\in\Sp:\, \langle u,x\rangle\geq\alpha_{d,N}\},
\]
where $\alpha_{d,N}\in(0,1)$ is chosen so that $\sigma_{d-1}(C(x))=1/N$. (Thus, $A_{d,N}=\alpha_{d,N}\sqrt{d}$.) Now take a standard Gaussian random vector $Y\sim\gamma_d$. In polar form, $Y=RU$ where $R=\|Y\|\geq 0$ and $U=Y/\|Y\|\in\Sp$; more specifically, $U\sim\sigma_{d-1}$ is uniform on the sphere $\Sp$, $R^2\sim\chi^2(d)$, and $U$ and $R$ are independent.

Define the cone $A(x)$ generated by the cap $C(x)$ to be its positive hull, i.e.,
\begin{equation}
    A(x)=\operatorname{pos}C(x)=\{\lambda u:\,u\in C(x),\lambda\geq0\}. 
\end{equation}
Then $U\in C(x)$ if and only if $Y\in A(x)$, since $Y=RU$ lies in the cone precisely when its direction $U$ lies in the cap. Thus, we can reduce the partial sphere covering problem to  a Gaussian coverage of cones:
\begin{equation}
    \sigma_{d-1}\left(\bigcup_{i=1}^{N/2}[C(x_i)\cup C(-x_i)]\right)=\Pr\left(Y\in\bigcup_{i=1}^{N/2}[A(x_i)\cup A(-x_i)]\right).
\end{equation}

For a fixed direction $x$, consider the (one-dimensional) random variable $\langle Y,x\rangle\sim N(0,1)$. Recall that we defined the Gaussian quantile $t_{d,N}$ by $\Phi(-t_{d,N})=1/N$. Then the halfspace
\[
H^+(x):=\{y\in\R^d:\,\langle y,x\rangle\geq t_{d,N}\}
\]
has $d$-dimensional Gaussian measure $1/N$. Since we are considering antipodal partial coverings, in  the Gaussian setting we consider the two-sided exceedance
\[
H^+(x)\cup H^+(-x)=\{y\in\R^d:\,|\langle y,x\rangle|\geq t_{d,N}\}.
\]
Note that $\gamma_d(H^+(x)\cup H^+(-x))=2/N$. By De Morgan's law, the complement is the origin-symmetric slab
\[
K(x):=[H^+(x)\cup H^+(-x)]^c=H^-(x)\cap H^-(-x)=\{y\in\R^d:\,|\langle y,x\rangle| \leq t_{d,N}\}
\]
where the last equality holds up to a set of Gaussian measure zero. Hence, $\gamma_d(K(x))=1-2/N$. 

Let $\eta_d:=t_{d,N}-\alpha_{d,N}\sqrt d$. By Lemma \ref{Lem:Caps}, $|\eta_d|=o(f(d)^{-1/2})$. Choose $s_d>0$ so that
\[
\sqrt{f(d)}\,s_d\to0,\qquad
s_d\sqrt{\frac d{f(d)}}\to\infty,\qquad
|\eta_d|=o(s_d).
\]
Such a choice is possible because $|\eta_d|=o(f(d)^{-1/2})$ and $\sqrt{f(d)/d}=o(f(d)^{-1/2})$. Define the trimmed Gaussian threshold $t_{d,N}':=t_{d,N}+s_d$.

By rotational invariance, it suffices to prove the following estimate for the cap centered at $e_d$.  Set
\[
A:=A(e_d)=\operatorname{pos}C(e_d).
\]
Writing $Y=(Y',Y_d)$ with $Y_d\sim N(0,1)$ and $Y'\sim N(0,I_{d-1})$ independent, we have
\[
A=\{(y',t):\,t\geq 0,\, \|y'\|\leq t\,r_{d,N}\}.
\]
Also, recall that
\[
r_{d,N}=\frac{\sqrt{1-\alpha_{d,N}^2}}{\alpha_{d,N}}
\]
is the radius of the cross-section of the cone at height $t=1$. By Fubini's theorem,
\begin{align}\label{fubini-1}
    \gamma_d(A\cap\{x_d\geq t_{d,N}'\}) = \frac{1}{\sqrt{2\pi}}\int_{t_{d,N}'}^\infty e^{-t^2/2}\Pr(\|Y'\|\leq t\cdot r_{d,N})\,dt.
\end{align}
We aim to show that
\begin{equation}\label{to-be-shown}
      \gamma_d(A\cap\{x_d\geq t_{d,N}'\})=(1-o(1))\Phi(-t_{d,N}').
\end{equation}
By the Pythagorean theorem, $r_{d,N}=\frac{\sqrt{1-\alpha_{d,N}^2}}{\alpha_{d,N}}$, so for any $t\geq t_{d,N}'$ we have
\begin{equation}\label{radius-est}
    t\cdot r_{d,N} = t\cdot \frac{\sqrt{1-\alpha_{d,N}^2}}{\alpha_{d,N}}\geq t_{d,N}'\cdot \frac{\sqrt{1-\alpha_{d,N}^2}}{\alpha_{d,N}}=\sqrt{d}\cdot \frac{t_{d,N}+s_d}{\alpha_{d,N}\sqrt{d}}\cdot\sqrt{1-\alpha_{d,N}^2}.
\end{equation}
Since $\alpha_{d,N}\sqrt d=t_{d,N}-\eta_d$, we have
\begin{equation}\label{middle-term}
    \frac{t_{d,N}+s_d}{\alpha_{d,N}\sqrt d}
    =\frac{t_{d,N}+s_d}{t_{d,N}-\eta_d}
    =1+\frac{s_d+\eta_d}{t_{d,N}-\eta_d}
    =1+(1+o(1))\frac{s_d}{t_{d,N}}
    =1+\Theta\!\left(\frac{s_d}{\sqrt{f(d)}}\right),
\end{equation}
where we used $|\eta_d|=o(s_d)$ and $t_{d,N}\asymp\sqrt{f(d)}$. Also, Lemma \ref{Lem:Caps} gives $\alpha_{d,N}^2=O(f(d)/d)$, and therefore
\begin{equation}\label{sqrt-term}
    \sqrt{1-\alpha_{d,N}^2}=1+O\left(\frac{f(d)}{d}\right).
\end{equation}
Substituting \eqref{middle-term} and \eqref{sqrt-term} into \eqref{radius-est}, we obtain, uniformly for all $t\geq t_{d,N}'$,
\begin{align*}
    t\cdot r_{d,N}
    &\geq \sqrt d\left(1+\Theta\!\left(\frac{s_d}{\sqrt{f(d)}}\right)\right)
    \left(1+O\left(\frac{f(d)}{d}\right)\right)
    =\sqrt d+\Theta\!\left(\frac{\sqrt d}{\sqrt{f(d)}}s_d\right)+o(1).
\end{align*}
Indeed, the error coming from $O(f(d)/d)$ contributes only $O(f(d)/\sqrt d)=o(1)$, while the product of $O(f(d)/d)$ with the positive term $\Theta(s_d/\sqrt{f(d)})$ is lower order. Since $s_d\sqrt{d/f(d)}\to\infty$, we obtain
\begin{equation}\label{eq:eq-16}
    t\cdot r_{d,N} \geq t_{d,N}'r_{d,N}=\sqrt{d}+\Theta\!\left(\frac{\sqrt{d}}{\sqrt{f(d)}}\cdot s_d\right)+o(1).
\end{equation}

Since $\|Y'\|\sim\chi_{d-1}$ is concentrated in a thin shell of width $O(1)$ around $\sqrt{d-1}$, we obtain the uniform bound 
\begin{equation}\label{uniform-bound-sup}
    \sup_{t\geq t_{d,N}'}\Pr(\|Y'\|>t\cdot r_{d,N})=o(1).
\end{equation}
We verify this. Recall that if $Z\sim N(0,I_m)$ and $f:\R^m\to\R$ is 1-Lipschitz, then for every $u\geq 0$, we have $\Pr(f(Z)\geq \E[f(Z)]+u)\leq Ce^{-u^2/2}$ (see, e.g., \cite[Theorem 5.2.2]{VershyninBook}). Note that with  $f(x)=\|x\|$, we have $|f(x)-f(y)|=\big|\|x\|-\|y\|\big|\leq \|x-y\|$, so $f$ is 1-Lipschitz. Hence, $\Pr(\|Z\|\geq\E[\|Z\|]+u)\leq Ce^{-u^2/2}$. Note that by the Cauchy-Schwarz inequality, $\E[\|Z\|]\leq \sqrt{\E[\|Z\|^2]}=\sqrt{m}$. Hence, $\{\|Z\|\geq\sqrt{m}+u\}\subset\{\|Z\|\geq\E[\|Z\|]+u\}$, which implies
\[
\Pr(\|Z\|\geq\sqrt{m}+u) \leq \Pr(\|Z\|\geq\E[\|Z\|]+u)\leq Ce^{-u^2/2}.
\]
Now choosing $Z=Y'$, $m=d-1$, and $u=u_d:=\Theta(\frac{\sqrt{d}}{\sqrt{f(d)}}\cdot s_d)+o(1)$, using also \eqref{eq:eq-16} we obtain 
\begin{equation}
 \sup_{t\geq t_{d,N}'}\Pr(\|Y'\|>t\cdot r_{d,N})\leq \Pr(\|Y'\|\geq\sqrt{d-1}+u_d)\leq Ce^{-u_d^2/2}=o(1)   
\end{equation}
where the last equality follows since $u_d\to\infty$. Combining \eqref{fubini-1} and \eqref{uniform-bound-sup}, we obtain \eqref{to-be-shown}:
\begin{equation}\label{cone-2}
    \gamma_d(A\cap\{x_d\geq t_{d,N}'\})=(1-o(1))\cdot\frac{1}{\sqrt{2\pi}}\int_{t_{d,N}'}^\infty e^{-t^2/2}\,dt=(1-o(1))\Phi(-t_{d,N}').
\end{equation}

In the next step, we prove that
\begin{equation}\label{Phi-cdf-tbd}
    \Phi(-t_{d,N}')=(1-o(1))\Phi(-t_{d,N})=\frac{1+o(1)}{N}.
\end{equation}
Define the Mills ratio $R(t):=\Phi(-t)/\phi(t)$ where $\phi(t)=(2\pi)^{-1/2}e^{-t^2/2}$. Then by the classical estimate (see, e.g., \cite{VershyninBook}) $\frac{1}{t}(1-\frac{1}{t^2})\leq R(t)\leq \frac{1}{t}$, we have $R(t)=\frac{1}{t}(1+\epsilon(t))$ with $|\epsilon(t)|\leq C/t^2$ for large $t$, and thus $\Phi(-t)=(1+\epsilon(t))\cdot\frac{\phi(t)}{t}$. Applying this estimate twice with $t=t_{d,N}'=t_{d,N}+s_d$ and $t=t_{d,N}$, we obtain
\begin{align*}
\frac{\Phi(-t_{d,N}')}{\Phi(-t_{d,N})}&=\frac{\phi(t_{d,N}+s_d)}{\phi(t_{d,N})}\cdot\frac{t_{d,N}}{t_{d,N}+s_d}\cdot\frac{1+\epsilon(t_{d,N}+s_d)}{1+\epsilon(t_{d,N})}\\
&=\exp\left(-t_{d,N}s_d-\frac{s_d^2}{2}\right)\cdot\frac{t_{d,N}}{t_{d,N}+s_d}\cdot\frac{1+\epsilon(t_{d,N}+s_d)}{1+\epsilon(t_{d,N})}.
\end{align*}
Taking logarithms, we get
\begin{align*}
    \ln\left(\frac{\Phi(-t_{d,N}')}{\Phi(-t_{d,N})}\right)=-t_{d,N}s_d-\frac{s_d^2}{2}+\ln\left(\frac{t_{d,N}}{t_{d,N}+s_d}\right)+\ln\left(\frac{1+\epsilon(t_{d,N}+s_d)}{1+\epsilon(t_{d,N})}\right).
\end{align*}
Since $\sqrt{f(d)}\cdot s_d\to 0$  and  $t_{d,N}\sim\sqrt{2f(d)}$, we have $s_d/t_{d,N}\to 0$. Hence,
\begin{align*}
    \ln\left(\frac{t_{d,N}}{t_{d,N}+s_d}\right)=-\ln\left(1+\frac{s_d}{t_{d,N}}\right)=-\frac{s_d}{t_{d,N}}+O\left(\frac{s_d^2}{t_{d,N}^2}\right)=o(1).
\end{align*}
Since $|\epsilon(t)|\leq C/t^2$, we have $\epsilon(t)\to 0$ as $t\to\infty$. Thus, for large $t$, $\ln(1+\epsilon(t))=\epsilon(t)+O(\epsilon(t)^2)$, and hence 
\[
\ln\left(\frac{1+\epsilon(t_{d,N}+s_d)}{1+\epsilon(t_{d,N})}\right)=\ln(1+\epsilon(t_{d,N}+s_d))-\ln(1+\epsilon(t_{d,N}))=O((t_{d,N}+s_d)^{-2})+O(t_{d,N}^{-2})=o(1).
\]
Therefore,
\begin{align*}
    \ln\left(\frac{\Phi(-t_{d,N}')}{\Phi(-t_{d,N})}\right)=-t_{d,N}s_d-\frac{s_d^2}{2}+o(1).
\end{align*}
Finally, since $t_{d,N}\sim\sqrt{2f(d)}$ and $\sqrt{f(d)}\cdot s_d\to 0$, we have $t_{d,N}s_d\to 0$. Moreover, $s_d\to 0$ follows from the hypotheses on $s_d$. Thus,
\begin{align*}
    \frac{\Phi(-t_{d,N}')}{\Phi(-t_{d,N})}=\exp\left(-t_{d,N}s_d-\frac{s_d^2}{2}+o(1)\right)=1+o(1).
\end{align*}
This proves \eqref{Phi-cdf-tbd}. Combining \eqref{cone-2} and \eqref{Phi-cdf-tbd}, we have thus shown that 
\begin{equation}
    \gamma_d(A\cap\{x_d\geq t_{d,N}'\})=\frac{1+o(1)}{N}.
\end{equation}
Let us note here that since $\gamma_d(A)=1/N$, we have $\gamma_d(A\cap\{x_d<t_{d,N}'\}) = o(1/N)$.

In the final step of the proof, we replace the cones by slab complements. Take any $x_1,\ldots,x_{N/2}\in\Sp$, and consider the ``antipodal'' cones $A(x_i)\cup A(-x_i)$.  Define the symmetric slab $K_i:=\{y:\,|\langle y,x_i\rangle|\leq t_{d,N}'\}$. Note that up to a set of Gaussian measure zero, for each $i$ we have $K_i^c=\{y:\,|\langle y,x_i\rangle|\geq t_{d,N}'\}$ and
\begin{align*}
    [A(x_i)\cup A(-x_i)] \subset K_i^c\cup\left([A(x_i)\cup A(-x_i)]\cap\{y:\,|\langle y,x_i\rangle|<t_{d,N}'\}\right).
\end{align*}
Taking the union over $i$, by the union bound and rotation invariance we obtain
\begin{align*}
    \gamma_d\left(\bigcup_{i=1}^{N/2}[A(x_i)\cup A(-x_i)]\right) &\leq \gamma_d\left(\bigcup_{i=1}^{N/2}K_i^c\right)+\sum_{i=1}^{N/2}\gamma_d\left([A(x_i)\cup A(-x_i)]\cap\{y:\,|\langle y,x_i\rangle|<t_{d,N}'\}\right)\\
    &\leq \gamma_d\left(\bigcup_{i=1}^{N/2}K_i^c\right)+\sum_{i=1}^{N/2}o(1/N)=\gamma_d\left(\bigcup_{i=1}^{N/2}K_i^c\right)+o(1).
\end{align*}
Moreover,
\[
\gamma_d(K_i^c)=\Pr(|\langle Y,x_i\rangle|\geq t_{d,N}')=2\Phi(-t_{d,N}')=\frac{2+o(1)}{N},
\]
so $\gamma_d(K_i)=1-\frac{2+o(1)}{N}$. Finally, by De Morgan's law and \v{S}id\'ak's lemma,
\begin{align*}
\gamma_d\left(\bigcup_{i=1}^{N/2}K_i^c\right)=1-\gamma_d\left(\bigcap_{i=1}^{N/2}K_i\right)\leq 1-\prod_{i=1}^{N/2}\gamma_d(K_i)=1-\left(1-\frac{2+o(1)}{N}\right)^{\frac{N}{2}}=1-e^{-1}+o(1).
\end{align*}
This completes the proof of the upper bound. \qed

	\subsection{A Technical Lemma and the Proof of Lemma \ref{Lem:Caps}}
    
 The estimate for $t_{d,N}$ in Lemma \ref{Lem:Caps} follows from standard estimates of order statistics (cf. \cite{LeadbetterLindgrenRootzen1983}). We present a slightly more general result below. 
    	\begin{lemma}\label{tdN-general}
		Let $t_{d,N}$ be such that $\Phi(-t_{d,N})=1/N$. Then, 
		\[
		t_{d,N}=\sqrt{2\ln N}\left[1-\frac{\ln\ln N}{4\ln N}-\frac{\ln(4\pi)}{4\ln N}+o\left(\frac{1}{\ln N}\right)\right].
		\]
	\end{lemma}

    We are now in a position to prove Lemma \ref{Lem:Caps}.

    	\vspace{2mm}
    
	\begin{proof}[Proof of Lemma \ref{Lem:Caps}] 
Let  $N:=N(d)\to\infty$, where $f(d)=\ln N(d)$. Assume that $f(d)\to\infty$ and $f(d)/\sqrt{d}\to 0$ as $d\to\infty$. Recall that $t_{d,N}>0$ is defined by $\Phi(-t_{d,N})=N^{-1}=e^{-f(d)}$. The estimate for $t_{d,N}$ follows directly from Lemma \ref{tdN-general} with $N=e^{f(d)}$ (which is justified since $f(d)\to\infty$ and $N=e^{f(d)}\to\infty$ as $d\to\infty$). Thus,
		\begin{align}\label{td-est}
			t_{d,N} = \sqrt{2f(d)}\left(1-\frac{\ln(4\pi f(d))}{4f(d)}+o(f(d)^{-1})\right)\quad\text{as }d\to\infty.
		\end{align}
        \noindent In particular, $t_{d,N}\sim\sqrt{2f(d)}$.
		
		Next, we derive the estimate		
        \[
		\delta_{d,N}:= |t_{d,N} - A_{d,N}| = \frac{f(d)^{3/2}}{\sqrt{2}\,d}
		+o\!\left(\frac{1}{\sqrt{f(d)}}+\frac{f(d)^{3/2}}{d}\right) = o(1/\sqrt{f(d)}).
		\]
Recall that $\alpha_{d,N}\in(0,1)$ is defined by
\[
\sigma_{d-1}(\{u\in\Sp:\,\langle u,e_d\rangle\geq \alpha_{d,N}\})=e^{-f(d)},
\]
and set $A_{d,N}:=\alpha_{d,N}\sqrt d$. Using the formula for the volume of the spherical cap, we get
\begin{equation}\label{eq:lemma-2-main-eqn}
e^{-f(d)}=C(d)\int_{\alpha_{d,N}}^1 (1-x^2)^{\frac{d-3}{2}}\,dx,
\end{equation}
where
\[
C(d):=\frac{\vol_{d-2}(\partial B_{d-1})}{\vol_{d-1}(\partial B_d)}
=\frac{\Gamma(d/2)}{\sqrt\pi\,\Gamma((d-1)/2)}.
\]
By Stirling's formula,
\[
C(d)=\sqrt{\frac{d}{2\pi}}
\left(1-\frac{3}{4d}+O(d^{-2})\right).
\]
Hence,
\begin{equation}\label{eq:log-Cd-lemma-2}
\ln C(d)
=\frac{1}{2}\ln d-\frac{1}{2}\ln(2\pi)-\frac{3}{4d}+O(d^{-2}).
\end{equation}

Next, we record two elementary consequences. First, we claim that $\alpha_{d,N}\to 0$ as $d\to\infty$. 
Indeed, suppose that, along some subsequence, 
$\alpha_{d,N}\geq \varepsilon>0$. On one hand, by \cite[Proposition 5.1,
Equation (5.5)]{aubrun2017alice}, we have
\[
\begin{aligned}
e^{-f(d)}
&=\sigma_{d-1}\bigl(\{u\in\Sp:
\langle u,e_d\rangle\geq\alpha_{d,N}\}\bigr)\\
&\leq \sigma_{d-1}\bigl(\{u\in\Sp:
\langle u,e_d\rangle\geq\varepsilon\}\bigr)\\
&\leq \frac{1}{2}
\exp\left(-\frac{d}{2}(\arcsin\varepsilon)^2\right)
\leq \frac{1}{2}e^{-d\varepsilon^2/2}.
\end{aligned}
\]
On the other hand, $e^{-f(d)}=e^{-o(d)}$, which yields a contradiction for all sufficiently large $d$.

Second, we claim that 
\[
A_{d,N}=\alpha_{d,N}\sqrt d\longrightarrow\infty\quad\text{as}\quad d\to\infty.
\]
Suppose not.  Then along some subsequence, we have $A_{d,N}\leq M$ for some $M>0$.  Since
$\alpha_{d,N}\leq M/\sqrt d$, for all sufficiently large $d$ we have
\begin{equation}\label{eq:lower-bound-lemma-2-proof}
e^{-f(d)}
\geq C(d)\int_{M/\sqrt{d}}^{(M+1)/\sqrt{d}}
(1-x^2)^{\frac{d-3}{2}}\,dx.
\end{equation}
Using $C(d)\sim\sqrt{d/(2\pi)}$ and
\[
\left(1-\frac{(M+1)^2}{d}\right)^{\frac{d-3}{2}}\longrightarrow
e^{-(M+1)^2/2}\quad\text{as}\quad d\to\infty,
\]
the right-hand side of \eqref{eq:lower-bound-lemma-2-proof} is bounded below by a positive constant depending only on $M$.  This contradicts our assumption that $e^{-f(d)}\to 0$ as $d\to\infty$.  Thus
$A_{d,N}\to\infty$, and in particular
\begin{equation}\label{eq:Adn-squared-Lemma-2}
A_{d,N}^2=d\alpha_{d,N}^2\longrightarrow\infty\quad\text{as}\quad d\to\infty.
\end{equation}

Next, we estimate the integral \eqref{eq:lemma-2-main-eqn}. Define the function
\[
I_d(\alpha):=\int_\alpha^1 (1-x^2)^{\frac{d-3}{2}}\,dx.
\]
Then \eqref{eq:lemma-2-main-eqn} becomes $e^{-f(d)}=C(d)I_d(\alpha_{d,N})$. 
For $m=(d-3)/2$ and $g(x)=(1-x^2)^m$, we have
\[
g'(x)=-\left(\frac{2mx}{1-x^2}\right)g(x).
\]
Thus,
\[
g(x)=-\left(\frac{1-x^2}{2mx}\right)g'(x).
\]
Integrating by parts on $[\alpha,1]$, with the endpoint at $x=1$ understood as a limit, we obtain
\begin{align*}
I_d(\alpha)&=\int_\alpha^1 g(x)\,dx=-\int_\alpha^1\left(\frac{1-x^2}{2mx}\right)g'(x)\,dx\\
&=-\frac{1}{2m}\left(\frac{1-x^2}{x}\cdot g(x)\bigg|_\alpha^1+\int_\alpha^1 g(x)\left(\frac{1+x^2}{x^2}\right)dx\right)
=\frac{(1-\alpha^2)^{m+1}}{2m\alpha}
-\frac{1}{2m}\int_\alpha^1
\frac{1+x^2}{x^2}(1-x^2)^m\,dx.
\end{align*}
Set $L_d(\alpha):=\frac{(1-\alpha^2)^{m+1}}{2m\alpha}$ and
\[
J_d(\alpha):=\frac{1}{2m}\int_\alpha^1
\frac{1+x^2}{x^2}(1-x^2)^m\,dx.
\]
Note that $I_d(\alpha)=L_d(\alpha)-J_d(\alpha)$. Since $\alpha\leq x\leq 1$, we have $\frac{1+x^2}{x^2}\leq\frac{2}{\alpha^2}$ for all $x\in[\alpha,1]$. Thus,
\[
0\leq J_d(\alpha)\leq
\frac{I_d(\alpha)}{m\alpha^2}=\varepsilon_d I_d(\alpha)
\]
where $\varepsilon_d:=(m\alpha^2)^{-1}$. Hence,
\[
L_d(\alpha)-\varepsilon_d I_d(\alpha)\leq I_d(\alpha)\leq L_d(\alpha).
\]
The left inequality implies $L_d(\alpha)\leq(1+\varepsilon_d)I_d(\alpha)$, and since $J_d(\alpha)\geq 0$ we have $I_d(\alpha)=L_d(\alpha)-J_d(\alpha)\leq L_d(\alpha)$. Combining these inequalities, we derive that
\[
\frac{L_d(\alpha)}{1+\varepsilon_d}\leq I_d(\alpha)\leq L_d(\alpha).
\]

Now let $\alpha=\alpha_{d,N}$. By \eqref{eq:Adn-squared-Lemma-2}, we have $\varepsilon_d^{-1}=m\alpha_{d,N}^2\to\infty$ as $d\to\infty$, which implies $\varepsilon_d\to 0$ as $d\to\infty$. Thus,
\[
\frac{1}{1+\varepsilon_d}=1+O(\varepsilon_d)=1+O\left(\frac{1}{d\alpha_{d,N}^2}\right)=1+O(A_{d,N}^{-2})
\]
and hence
\[
I_d(\alpha_{d,N})
=L_d(\alpha_{d,N})\left(1+O(A_{d,N}^{-2})\right)=\frac{(1-\alpha_{d,N}^2)^{\frac{d-1}{2}}}
     {(d-3)\alpha_{d,N}}
\left(1+O(A_{d,N}^{-2})\right).
\]
Therefore,
\begin{equation}\label{eq:Id-final-lemma-2}
e^{-f(d)}
=C(d)\cdot\frac{(1-\alpha_{d,N}^2)^{\frac{d-1}{2}}}
     {(d-3)\alpha_{d,N}}
\left(1+O(A_{d,N}^{-2})\right).
\end{equation}
Taking logarithms of both sides, we get
\[
f(d)
=-\frac{d-1}{2}\ln(1-\alpha_{d,N}^2)
+\ln(d-3)+\ln\alpha_{d,N}
-\ln C(d)+\ln(1+O(A_{d,N}^{-2})).
\]
By \eqref{eq:Adn-squared-Lemma-2}, we have $A_{d,N}^{-2}\to 0$ as $d\to\infty$, so $\ln(1+O(A_{d,N}^{-2}))=O(A_{d,N}^{-2})$. Hence,
\[
f(d)
=-\frac{d-1}{2}\ln(1-\alpha_{d,N}^2)
+\ln(d-3)+\ln\alpha_{d,N}
-\ln C(d)+O(A_{d,N}^{-2}).
\]
Substituting $\alpha_{d,N}=A_{d,N}/\sqrt d$ and using \eqref{eq:log-Cd-lemma-2}, we get  
\[
\ln(d-3)+\ln\alpha_{d,N}-\ln C(d)
=\ln A_{d,N}+\frac{1}{2}\ln(2\pi)+O(d^{-1}),
\]
which yields
\begin{equation}\label{eq:lemma-2-fd-lower-bd}
f(d)
=-\frac{d-1}{2}\ln\left(1-\frac{A_{d,N}^2}{d}\right)+\ln A_{d,N}
+\frac{1}{2}\ln(2\pi)
+O(A_{d,N}^{-2})+O(d^{-1}).
\end{equation}

We now show that $A_{d,N}^2=O(f(d))$.  Using the inequality
$-\ln(1-x)\geq x$ for $0<x<1$, we obtain
\[
f(d)\geq
\frac{d-1}{2}\frac{A_{d,N}^2}{d}
+\ln A_{d,N}
+\frac{1}{2}\ln(2\pi)
+O(A_{d,N}^{-2})+O(d^{-1})\geq \frac{d-1}{2d}\cdot A_{d,N}^2
+\ln A_{d,N}-C
\]
for some positive absolute constant $C$. As, by \eqref{eq:Adn-squared-Lemma-2}, we have $A_{d,N}\to\infty$, the logarithmic term is nonnegative for all sufficiently
large $d$. Moreover, $\frac{d-1}{2d}\geq \frac{1}{4}$ for all $d\geq 2$, so for all sufficiently large $d$ we have
\[
f(d) \geq \frac{A_{d,N}^2}{4}-C,
\]
i.e., $A_{d,N}^2\leq 4(f(d)+C)$. This proves that $A_{d,N}^2=O(f(d))$. In particular,
\[
\frac{A_{d,N}^2}{d}=O\left(\frac{f(d)}{d}\right)\longrightarrow 0\quad\text{as }d\to\infty,
\]
and
\[
\frac{A_{d,N}^6}{d^2}=O\left(\frac{f(d)^3}{d^2}\right)=o(1),
\]
because $f(d)=o(\sqrt{d})$. Expanding the logarithm in \eqref{eq:lemma-2-fd-lower-bd}, we derive that
\[
-\frac{d-1}{2}\ln\left(1-\frac{A_{d,N}^2}{d}\right)
=\frac{A_{d,N}^2}{2}
+\frac{A_{d,N}^4}{4d}
+O\left(\frac{A_{d,N}^2}{d}\right)
+O\left(\frac{A_{d,N}^6}{d^2}\right).
\]
As shown above, the two error terms are $o(1)$.  Hence,
\begin{equation}\label{eq:lemma-2-fd-est-final}
f(d)=\frac{A_{d,N}^2}{2}+\ln A_{d,N}
+\frac{1}{2}\ln(2\pi)
+\frac{A_{d,N}^4}{4d}+o(1).
\end{equation}

Define the function
\[
\Psi_d(A):=
\frac{A^2}{2}+\ln A+\frac{1}{2}\ln(2\pi)
+\frac{A^4}{4d}.
\]
Note that  $\Psi_d'(A)=A+A^{-1}+A^3/d>0$ since $A>0$. Set $T_d:=\sqrt{2f(d)}$.  The preceding estimate \eqref{eq:lemma-2-fd-est-final} says
\[
\Psi_d(A_{d,N})=f(d)+o(1).
\]
Thus,
\[
\Psi_d(A_{d,N})-\Psi_d(T_d)
=f(d)-\Psi_d(T_d)+o(1).
\]
By the mean value theorem, there exists $\xi_d$ between $A_{d,N}$
and $T_d$ such that
\[
\Psi_d'(\xi_d)(A_{d,N}-T_d)
=f(d)-\Psi_d(T_d)+o(1).
\]
We first note that $A_{d,N}\sim T_d$. Indeed, from
\[
f(d)=\frac{A_{d,N}^2}{2}+O(\ln A_{d,N})+O(A_{d,N}^4/d)+o(1)
\]
and $A_{d,N}^2=O(f(d))$, it follows that $A_{d,N}^2=2f(d)+o(f(d))$. Hence $A_{d,N}\sim T_d$, and therefore $\xi_d\sim T_d$. Consequently,
\[
\Psi_d'(\xi_d)=\sqrt{2f(d)}(1+o(1)).
\]
Note that this implies
\[
A_{d,N}-T_d
=O\left(\frac{\ln f(d)+f(d)^2/d+1}{\sqrt{f(d)}}\right)
=O\left(\frac{\ln f(d)}{\sqrt{f(d)}}\right)
=o(T_d).
\] 

Applying Taylor's formula at $A=T_d$, we derive that
\[
\Psi_d(A_{d,N})
=\Psi_d(T_d)
+\Psi_d'(T_d)(A_{d,N}-T_d)
+O((A_{d,N}-T_d)^2),
\]
because $\Psi_d''(A)=1-A^{-2}+3A^2/d$ is bounded on the interval between $A_{d,N}$ and $T_d$. Since
\[
A_{d,N}-T_d=O\left(\frac{\ln f(d)}{\sqrt{f(d)}}\right),
\]
we have
\[
(A_{d,N}-T_d)^2
=O\left(\frac{(\ln f(d))^2}{f(d)}\right).
\]
Dividing this error by $\Psi_d'(T_d)\sim \sqrt{2f(d)}$ gives an error of
\[
O\left(\frac{(\ln f(d))^2}{f(d)^{3/2}}\right)
=o(f(d)^{-1/2}).
\]
Therefore,
\begin{equation}\label{eq:lemma-2-Taylor}
A_{d,N}=T_d-\frac{\Psi_d(T_d)-f(d)}{\Psi_d'(T_d)}+o(f(d)^{-1/2}).
\end{equation}

Set $f=f(d)$, $L=\ln(4\pi f)$, and $\epsilon_d:=\frac{1}{2f}+\frac{2f}{d}$. 
Then $\epsilon_d\to 0$ as $d\to\infty$, and
\[
\frac{\Psi_d(T_d)-f}{\Psi_d'(T_d)}
=\frac{\frac{1}{2}L+f^2/d}{\sqrt{2f}(1+\epsilon_d)}.
\]
Since $(1+\epsilon_d)^{-1}=1+O(\epsilon_d)$, we have
\[
\frac{\Psi_d(T_d)-f}{\Psi_d'(T_d)}
=\left(\frac{L}{2\sqrt{2f}}
+\frac{f^{3/2}}{\sqrt2\,d}\right)
(1+O(\epsilon_d)).
\]
The error term is
\begin{align*}
\left(\frac{L}{2\sqrt{2f}}
+\frac{f^{3/2}}{\sqrt2\,d}\right)
O(\epsilon_d)&=O\left[
\left(\frac{L}{\sqrt f}+\frac{f^{3/2}}{d}\right)
\left(\frac{1}{f}+\frac{f}{d}\right)
\right]=O\left(\frac{L}{f^{3/2}}+\frac{L\sqrt{f}}{d}
+\frac{\sqrt{f}}{d}+\frac{f^{5/2}}{d^2}\right)
=o(f^{-1/2}),
\end{align*}
where we used $L=o(f)$ and $f=o(\sqrt{d})$. Therefore,
\[
\frac{\Psi_d(T_d)-f}{\Psi_d'(T_d)}
=\frac{L}{2\sqrt{2f}}+\frac{f^{3/2}}{\sqrt{2}\,d}+o(f^{-1/2}).
\]
Substituting this into the Taylor estimate \eqref{eq:lemma-2-Taylor} and using $T_d=\sqrt{2f}$, we get
\[
A_{d,N}=\sqrt{2f}-\frac{L}{2\sqrt{2f}}
-\frac{f^{3/2}}{\sqrt{2}\,d}+o(f^{-1/2}).
\]
This says that
\[
\alpha_{d,N}\sqrt{d}
=\sqrt{2f(d)}
-\frac{\ln(4\pi f(d))}{2\sqrt{2f(d)}}
-\frac{f(d)^{3/2}}{\sqrt{2}\,d}+o(f(d)^{-1/2}).
\]
By Lemma~\ref{tdN-general},
\[
t_{d,N}=\sqrt{2f(d)}-\frac{\ln(4\pi f(d))}{2\sqrt{2f(d)}}+o(f(d)^{-1/2}).
\]
Consequently,
\[
t_{d,N}-\alpha_{d,N}\sqrt{d}
=\frac{f(d)^{3/2}}{\sqrt{2}\,d}
+o(f(d)^{-1/2}),
\]
and therefore
\[
\delta_{d,N}
=
\left|t_{d,N}-\alpha_{d,N}\sqrt{d}\right|
=\frac{f(d)^{3/2}}{\sqrt{2}\,d}
+o(f(d)^{-1/2}).
\]
Since $f(d)=o(\sqrt d)$, this also implies $\delta_{d,N}=o(f(d)^{-1/2})$. This completes the proof of Lemma 2. 
	\end{proof}
    
\section*{Acknowledgments}
	This material is based upon work supported by the National Science Foundation under Grant No. DMS-1929284 while SH was in residence at the Institute for Computational and Experimental Research in Mathematics in Providence, RI, during the Harmonic Analysis and Convexity program. GK deeply acknowledges his invaluable mentor, Prof. Emanuel Milman, for the useful discussions and for suggesting key ideas in this proof, and Prof. Adam R. Klivans for sharing recent approximation results in the Gaussian space of convex bodies. We would also like to thank the anonymous referee for a careful reading of the manuscript and for helpful suggestions that improved the paper.
    
\bibliography{main}

@misc{ABMB-blog-post,
    author = {Aubrun, G. and Szarek, S.},
    title = {How optimal are random sphere coverings?},
    year ={2017},
    howpublished={\url{https://aliceandbobmeetbanach.wordpress.com/2017/11/02/how-optimal-are-random-sphere-coverings/}},
    note = {[Online; accessed 23-November-2024]}
  }

@inproceedings{klivans2008learning,
  title={Learning geometric concepts via Gaussian surface area},
  author={Klivans, A. R. and O'Donnell, R. and Servedio, R. A.},
  booktitle={2008 49th Annual IEEE Symposium on Foundations of Computer Science},
  pages={541--550},
  year={2008},
  organization={IEEE}
}

@inproceedings{nazarov2004maximal,
  title={On the maximal perimeter of a convex set in $\mathbb{R}^n$ with respect to a Gaussian measure},
  author={Nazarov, F.},
  booktitle={Geometric Aspects of Functional Analysis: Israel Seminar 2001-2002},
  pages={169--187},
  year={2004},
  organization={Springer}
}

@book{LeadbetterLindgrenRootzen1983,
  author    = {Leadbetter, M. R. and Lindgren, G. and Rootz{\'e}n, H.},
  title     = {Extreme Value Theory: An Introduction},
  publisher = {Springer},
  year      = {1983},
  series    = {Springer Series in Statistics}
}

@article{ball1993reverse,
  title={The reverse isoperimetric problem for {G}aussian measure},
  author={Ball, K.},
  journal={Discrete \& Computational Geometry},
  volume={10},
  number={4},
  pages={411--420},
  year={1993},
  publisher={Springer-Verlag Berlin, Heidelberg}
}

@inproceedings{de2024gaussian,
  title={Gaussian approximation of convex sets by intersections of halfspaces},
  author={De, A. and Nadimpalli, S. and Servedio, R. A.},
  booktitle={2024 IEEE 65th Annual Symposium on Foundations of Computer Science (FOCS)},
  pages={1911--1930},
  year={2024},
  organization={IEEE}
}

@article{milman2026gaussian,
  title={The {G}aussian {C}onjugate {R}ogers--{S}hephard {I}nequality},
  author={Milman, E. and Nakamura, S. and Tsuji, H.},
  journal={arXiv preprint arXiv:2602.07981},
  year={2026}
}

@book{alon2016probabilistic,
  title={The Probabilistic Method},
  author={Alon, N. and Spencer, J. H.},
  year={2016},
  publisher={John Wiley \& Sons},
    edition={4}
}

@book{AGA_Part1,
  author    = {Artstein-Avidan, S.  and Giannopoulos, A. and Milman, V. D.},
  title     = {Asymptotic Geometric Analysis, Part I},
  series    = {Mathematical Surveys and Monographs},
  volume    = {202},
  year      = {2015},
  publisher = {American Mathematical Society},
  address   = {Providence, RI},
  isbn      = {978-1-4704-2099-5},
  doi       = {10.1090/surv/202},
  url       = {https://bookstore.ams.org/surv-202},
}

@article{EFR, title={The amount of overlapping in partial coverings of space by equal spheres}, volume={11}, DOI={10.1112/S0025579300004393}, number={2}, journal={Mathematika}, author={Erdős, P. and Few, L. and Rogers, C. A.}, year={1964}, pages={171–184}}

@book{aubrun2017alice,
  title={Alice and Bob Meet Banach: The Interface of Asymptotic Geometric Analysis and Quantum Information Theory},
  author={Aubrun, G. and Szarek, S. J.},
  isbn={9781470434687},
  lccn={2017010894},
  issn={0885-4653},
  series={Mathematical surveys and monographs},
  url={https://books.google.com/books?id=FfWcxAEACAAJ},
  year={2017},
  publisher={American Mathematical Society}
}

@incollection{Ball-2001,
author={Ball, K.},
title={Convex {G}eometry and {F}unctional {A}nalysis},
booktitle={Handbook of the Geometry of Banach Spaces},
volume={1},
year={2001},
pages={161--193},
publisher={Elsevier},
editor={Johnson, W. B. and Lindenstrauss, J.},
address={Amsterdam}
}

@article{BH-2024,
author={Besau, F. and Hoehner, S.},
title={An intrinsic volume metric for the class of convex bodies in $\mathbb{R}^n$},
journal={Communications in Contemporary Mathematics},
year={2024},
volume={26},
number={3},
pages={article no. 2350006}
}

@article{BHK,
author={Besau, F. and Hoehner, S. and Kur, G.},
title={Intrinsic and {D}ual {V}olume {D}eviations of {C}onvex {B}odies and {P}olytopes},
journal={International Mathematics Research Notices},
volume={2021},
number={22},
pages={17456--17513},
year={2021}
}

@article{BL-1999,
author={B\"{o}r\"{o}czky, K. J. and  Ludwig, M.},
title={Approximation of convex bodies and a momentum lemma for power diagrams}, 
journal={Monatshefte f\"{u}r Mathematik},
volume={127},
year={1999}, 
pages={101--110}
}

@article{Boroczky2004,
author={B\"or\"oczky, K. J.},
title={Polytopal approximation bounding the number of $k$-faces},
journal={Journal of Approximation Theory},
volume={102},
year={2000}, 
pages={263--285} 
}

@incollection{Boroczky-Wintsche-2003,
author={B\"or\"oczky, K. J. and Wintsche, G.},
title={Covering the sphere by equal spherical balls},
booktitle={Discrete and Computational Geometry -- The Goodman-Pollack Festschrift},
editor={Aronov, B. and Bas\'u, S. and Sharir, M. and Pach, J.},
year={2003},
publisher={Springer-Verlag},
pages={237--253}
}

@article{Glazyrin,
title={Personal communication},
author={Glazyrin, A.},
year={March 2018},
}

@article{GW2018,
author={Grote, J. and Werner, E.},
year={2018},
title={Approximation of smooth convex bodies by random polytopes},
journal={Electronic Journal of Probability},
volume={23},
pages={1--21}
}

@article{GTW2021,
author={Grote, J. and Th\"ale, C. and Werner, E.},
title={Surface area deviation between smooth convex bodies and polytopes},
year={2021},
journal={Advances in Applied Mathematics},
volume={129},
pages={paper no. 102218}
}

@article{HK-DCG,
  author    = {Hoehner, S. and Kur, G.},
  title     = {A {C}oncentration {I}nequality for {R}andom {P}olytopes, {D}irichlet--{V}oronoi
               {T}iling {N}umbers and the {G}eometric {B}alls and {B}ins {P}roblem},
  journal   = {Discrete \& Computational Geometry},
  volume    = {65},
  number    = {3},
  pages     = {730--763},
  year      = {2021} 
}

@article{HSW,
  title={The {S}urface {A}rea {D}eviation of the {E}uclidean {B}all and a {P}olytope},
  author={Hoehner, S. and Sch\"{u}tt, C. and Werner, E.},
  journal={Journal of Theoretical Probability},
  pages={244--267},
  year={2018},
  volume={31},
  number={1},
  publisher={Springer}
}

@article{HSW-2025,
title={Approximation of the {E}uclidean ball by polytopes with a fixed number of $k$-faces},
author={Hoehner, S. and Sch\"utt, C. and Werner, E.},
year={2025},
journal={arXiv:2510.22771}
}

@article{Hoorfar-Hassani,
author={Hoorfar, A. and Hassani, M.}, title={Inequalities on the {L}ambert {W} {F}unction and {H}yperpower {F}unction}, journal={JIPAM}, 
volume={9}, 
number={2}, 
year={2008}
}

@article{KL-1978,
author={Kabatjanski\u{\i}, G. A. and Leven\v{s}te\u{\i}n, V. I.},
title={Bounds for packings on the sphere and in space},
journal={Problemy Pereda\v{c}i Informacii},
volume={14},
year={1978}, 
number={1}, 
pages={3--25}
}

@article{Kur2017,
title={Approximation of the {E}uclidean ball by polytopes with a restricted number of facets},
author={Kur, G.},
year={2020},
journal={Studia Mathematica},
pages={111--133},
volume={251},
}

@article{ludwig1999,
author={Ludwig, M.},
title={Asymptotic approximation of smooth convex bodies by general polytopes},
journal={Mathematika},
volume={46},
year={1999}, 
pages={103--125}
}

@article{LSW,
author={Ludwig, M. and Sch\"utt, C. and Werner, E. M.},
title={Approximation of the {E}uclidean ball by polytopes},
journal={Studia Mathematica},
volume={173},
year={2006}, 
pages={1--18}
}

@book{Matousek-2002,
author={Matou\v{s}ek, J.},
title={Lectures on Discrete Geometry},
publisher={Springer-Verlag},
year={2002}
}

@article{royen,
author={Royen, T.},
title={A simple proof of the Gaussian correlation conjecture extended to multivariate
gamma distributions}, 
journal={Far East Journal of Theoretical Statistics},
volume={48},
year={2014}, 
pages={139--145}
}

@book{SchneiderBook, 
place={Cambridge}, 
edition={2}, 
series={Encyclopedia of Mathematics and its Applications}, 
title={Convex Bodies: The Brunn–Minkowski Theory},  publisher={Cambridge University Press}, author={Schneider, R.}, 
year={2013}, 
volume={151},
collection={Encyclopedia of Mathematics and its Applications}}

@incollection{SW-affine-SA,
title = {Affine surface area},
booktitle = {Harmonic Analysis and Convexity},
author = {Sch\"utt, C. and Werner, E. M.},
editor = {Koldobsky, A. and Volberg, A.},
publisher = {De Gruyter},
address = {Berlin, Boston},
pages = {427--444},
year = {2023}
}

@book{VershyninBook, 
place={Cambridge}, 
series={Statistical and Probabilistic Mathematics}, 
title={High-Dimensional Probability: An Introduction with Applications to Data Science},  publisher={Cambridge University Press}, author={Vershynin, R.}, 
year={2018}, 
collection={Statistical and Probabilistic Mathematics}}

@article{Milman-GCI-IBL,
  author        = {Milman, E.},
  title         = {Gaussian {C}orrelation via {I}nverse {B}rascamp--{L}ieb},
  journal       = {Probability Theory and Related Fields},
  note          = {To appear},
  year          = {2025}
}

@article{ACS-Sidak-Khatri,
  author        = {Assouline, R. and Chor, A. and Sadovsky, S.},
  title         = {A refinement of the {{\v S}id{\'a}k--{K}hatri} inequality and a strong {G}aussian correlation conjecture},
  year          = {2024},
  journal={arXiv:2407.15684}
}

@article{Hoehner-Thale-CLT,
  author        = {Hoehner, S. and Th{\"a}le, C.},
  title         = {Central {L}imit {T}heorem for {R}andom {P}artial {S}phere {C}overings in {H}igh {D}imensions},
  year          = {2026},
  journal = {arXiv:2604.07711}
  }

\vspace{3mm}
	
	\noindent {\sc Department of Mathematics \& Computer Science, Longwood University, U.S.A.}
	
	\noindent {\it E-mail address:} {\tt hoehnersd@longwood.edu}
	
	\vspace{3mm}
	
	\noindent{\sc Departments of Computer Science and Mathematics, ETH Zürich, Switzerland}
	
	\noindent {\it E-mail address:} {\tt gil.kur@inf.ethz.ch}
    \appendix

\newpage
\section{Appendix}\label{S:A}
\subsection{Lower Bound}
	To prove the lower bound, we use random antipodal partial coverings. We follow the argument in \cite{HK-DCG}, with slight modifications to account for antipodality.  We include the details for the reader's convenience. 
	
	Choose a sequence of points $X_1,X_2,\ldots$ randomly and independently from $\Sp$ according to the uniform distribution, and for every even $N=N(d)$ satisfying the hypotheses of Theorem \ref{mainThm-corrected}, writing $f(d)=\ln N(d)$, consider the random variable 
	\[
	\mathbb{V}(X_1,\ldots,X_{N/2}):=\sigma_{d-1}\left(\bigcup_{i=1}^{N/2}[C(X_i)\cup C(-X_i)]\right).
	\]
	Each point $X_i$ uniquely determines a cap of the sphere centered at $X_i$ with normalized surface measure $1/N$. 
    By Fubini's theorem and the fact that the $X_i$ are independent and identically distributed, 
	\begin{align*}
		\max_{\pm x_1,\ldots,\pm x_{N/2}\in\Sp}&\left\{\sigma_{d-1}\left(\bigcup_{i=1}^{N/2} [C(x_i)\cup C(-x_i)]\right):\,\sigma_{d-1}(C(x_i))=1/N,\,i=1,\ldots,N/2\right\}\\
		&\geq \E\left[\mathbb{V}(X_1,\ldots,X_{N/2})\right]
		=\E\left[\int_{\Sp}\left(1-\mathbbm{1}_{\max\langle \pm X_i,x\rangle\leq \alpha_{d,N}}(x)\right)d\sigma_{d-1}(x)\right]\\
        &=
\int_{\Sp}
\mathbb{P}\left(
x\in\bigcup_{i=1}^{N/2}[C(X_i)\cup C(-X_i)]
\right)\,d\sigma_{d-1}(x)\\
&=1-\int_{\Sp}
\prod_{i=1}^{N/2}
\mathbb{P}\left(x\notin C(X_i)\cup C(-X_i)\right)
\,d\sigma_{d-1}(x).
	\end{align*}
In the inequality, we used the fact that some realization of the random configuration attains at least the expected covered measure.  For each fixed $x\in\Sp$, we have
\[
\mathbb{P}\left(x\in C(X_i)\cup C(-X_i)\right)=\frac{2}{N},
\]
because each of the two antipodal caps has measure $1/N$ and the boundary has measure zero. Hence,
\[
\E\left[\mathbb{V}(X_1,\ldots,X_{N/2})\right]
=1-\left(1-\frac{2}{N}\right)^{N/2}=1-e^{-1}+e^{-1}N^{-1}+O(N^{-2}).
\]
Since $e^{-f(d)}\to 0$ as $d\to\infty$, we obtain the lower bound
	\[
	\lim_{d\to\infty}\max_{\pm x_1,\ldots,\pm x_{N/2}\in\Sp}\left\{\sigma_{d-1}\left(\bigcup_{i=1}^{N/2} [C(x_i)\cup C(-x_i)]\right):\,\sigma_{d-1}(C(x_i))=1/N,\,i=1,\ldots,N/2\right\}\geq 1-e^{-1}.
	\]
    This proves the lower bound. \qed
    
\subsection{Proof of Lemma \ref{tdN-general}}
		For $z>0$, we use the following well-known estimate for the Mills ratio of a standard normal random variable (see, e.g.,  \cite[Proposition 2.1.2]{VershyninBook}): \begin{equation}\label{gaussian-tails-Mills}
			\left(\frac{1}{z}-\frac{1}{z^3}\right)\frac{1}{\sqrt{2\pi}}\, e^{-z^2/2}\leq 
			\Phi(-z)\leq \frac{1}{z}\cdot\frac{1}{\sqrt{2\pi}}\, e^{-z^2/2}.
		\end{equation}  
		Setting $z:=t_{d,N}$ in \eqref{gaussian-tails-Mills}, we get
		\begin{equation}\label{mills-to-do}
			\frac{1}{z}\left(1-\frac{1}{z^2}\right)e^{-z^2/2}\leq \frac{\sqrt{2\pi}}{N}\leq \frac{1}{z}\cdot e^{-z^2/2}.
		\end{equation}
		Let us first consider the equation corresponding to the upper bound. Taking logarithms, we get
		\[
		\ln\sqrt{2\pi}-\ln N =-\ln z-\frac{z^2}{2}.
		\]
		Multiplying both sides by $-2$, we get $2\ln N-\ln(2\pi) =\ln(z^2)+z^2$. Let $y:=z^2$ and $C:=\ln(2\pi)$. Then $y+\ln y=2\ln N-C$. Exponentiating both sides of the last equation, we obtain $ye^y=N^2/e^C$. 
		By the definition of the Lambert W function, this implies $y = W\left(\frac{N^2}{e^C}\right)=W\left(\frac{N^2}{2\pi}\right)$. 
		
		For large $x$, we have the following asymptotic estimate (see, e.g., \cite[Theorem 2.7]{Hoorfar-Hassani}): \[
		W(x) = \ln x-\ln(\ln x)+O\left(\frac{\ln(\ln x)}{\ln x}\right).
		\]
		Take $x=x(N):=\frac{N^2}{2\pi}$. Note that $x(N)\to\infty$ as $N\to\infty$, and $\ln x(N) =2\ln N-\ln(2\pi)$. 
		Thus,
		\begin{align*}
			\ln(\ln x(N)) &=\ln\left(2\ln N-\ln(2\pi)\right)
			=\ln\left(2\ln N\left[1-\frac{\ln(2\pi)}{2\ln N}\right]\right)\\
			&=\ln(2\ln N)+\ln\left(1-\frac{\ln(2\pi)}{2\ln N}\right)
			=\ln(2\ln N)-\frac{\ln(2\pi)}{2\ln N}+O\left(\frac{1}{(\ln N)^2}\right)
		\end{align*}
		where in the last line we used $\ln(1-\epsilon)=-\epsilon+O(\epsilon^2)$. Hence,
		\begin{align*}
			y &=W(x(N)) =\ln x(N)-\ln(\ln x(N))+O\left(\frac{\ln(\ln x(N))}{\ln x(N)}\right)\\
			&=2\ln N-\ln(2\ln N)-\ln(2\pi)+\frac{\ln(2\pi)}{2\ln N}-O\left(\frac{1}{(\ln N)^2}\right)\\
			&=2\ln N\left[1-\left(\frac{\ln\ln N}{2\ln N}+\frac{\ln(2\pi)+\ln 2}{2\ln N}-\frac{\ln(2\pi)}{2(\ln N)^2}+O\left(\frac{1}{(\ln N)^3}\right)\right)\right].
		\end{align*}
		Therefore,
		\begin{align*}
			t_{d,N}=z&=\sqrt{y}=\sqrt{2\ln N}\sqrt{1-\left(\frac{\ln(\ln N)}{2\ln N}+\frac{\ln(4\pi)}{2\ln N}-\frac{\ln(2\pi)}{2(\ln N)^2}+O\left(\frac{1}{(\ln N)^3}\right)\right)}\\
			&=\sqrt{2\ln N}\left[1-\frac{\ln\ln N}{4\ln N}-\frac{\ln(4\pi)}{4\ln N}+O\left(\frac{\ln\ln N}{(\ln N)^2}\right)\right].
		\end{align*}
		
		The lower bound in \eqref{mills-to-do} is handled in the same way. The additional factor $1-z^{-2}$ contributes only
\[
\ln(1-z^{-2})=O(z^{-2})=O((\ln N)^{-1})
\]
to the logarithmic equation for $z^2$, and therefore changes $z$ only by
$O((\ln N)^{-3/2})$. This may also be absorbed into an $o((\ln N)^{-1})$ relative error. Hence, the same asymptotic expansion holds for $t_{d,N}$ in this case as well. \qed
\end{document}